\documentclass[12pt]{amsart}
\usepackage{amssymb}
\usepackage{array}

\usepackage{tikz}
\newtheorem{Teo}{Theorem}[section]
\newtheorem*{theorem*}{Theorem}
\newtheorem{Lema}[Teo]{Lemma}
\newtheorem{Obs}[Teo]{Remark}

\newtheorem{Def}[Teo]{Definition}

\newtheorem{Cor}[Teo]{Corollary}
\newtheorem{Que}[Teo]{Question}
\newtheorem{Prop}[Teo]{Proposition}

\newcommand{\VR}{\mathcal{O}}

\newtheorem{lem-def}[Teo]{Lemma-Definition}

\renewenvironment{proof}{{\bfseries Proof.}}{\qed}
\newcommand{\Lra}{\Longrightarrow}
\newcommand{\R}{\mathbb R}

\newcommand{\N}{\mathbb N}
\newcommand{\Z}{\mathbb Z}
\newcommand{\Q}{\mathbb Q}

\newcommand{\F}{\mathbb F}

\def\op{\operatorname}

\def\ars#1{\renewcommand\arraystretch{#1}}

\def\bs{\vskip.5cm}

\def\ps{\partial_s}

\def\dep{\op{\mbox{\rm\small depth}}}

\def\diso{\lower.4ex\hbox{$\downarrow$}\raise.4ex\hbox{\mbox{\scriptsize
$\wr$}}}

\def\dta{\delta}

\def\e{\medskip}

\newcommand{\Lr}{\ \Longrightarrow\ }

\def\ep{\epsilon}

\def\ep#1{\exp(\Pi i#1)}
\def\ep{\epsilon}

\def\g{\Gamma}
\def\ga{\gamma}

\def\gal{\op{Gal}}

\def\gi{\g_{\infty}}

\def\ism{\lower.3ex\hbox{\ars{.08}$\begin{array}{c}\,\to\\\mbox{\tiny $\sim\,$}\end{array}$}}
\def\iso{\ \lower.3ex\hbox{\ars{.08}$\begin{array}{c}\lra\\\mbox{\tiny $\sim\,$}\end{array}$}\ }

\def\kb{\overline{K}}

\def\kx{K[x]}

\def\lg{l\raise.6ex\hbox to.2em{\hss.\hss}l}

\def\lra{\,\longrightarrow\,}

\def\orb{\hbox to  .3em{$\backslash$}\backslash}

\def\sii{\ \Longleftrightarrow\ }

\def\sii{\quad\Longleftrightarrow\quad}

\def\sub{\subseteq}

\def\t{\theta}

\newcounter{cs}
\stepcounter{cs}
\newcommand{\casos}{\begin{itemize}}
\newcommand{\fcasos}{\end{itemize}\setcounter{cs}{1}}

\newfont{\tit}{cmr12 scaled \magstep3}

\title{On the distances of an element to its conjugates}
\makeatletter
\@namedef{subjclassname@2010}{%
  \textup{2010} Mathematics Subject Classification}
\subjclass[2010]{Primary 13A18; Secondary 12J20, 13J10, 14E15}

\author[Josnei Novacoski]{Josnei Novacoski}
\address{Departamento de Matem\'{a}tica,         Universidade Federal de S\~ao Carlos, Rod. Washington Luís, 235, 13565--905, S\~ao Carlos -SP, Brazil}
\email{josnei@ufscar.br}

\thanks{During the realization of this project the author was supported by a grant from Funda\c{c}\~ao de Amparo \`a Pesquisa do Estado de S\~ao Paulo (process number 2024/08989-6) and a grant from Conselho Nacional de Desenvolvimento Cient\'ifico e Tecnol\'ogico (process number 303215/2022-4).}

\keywords{Artin-Schreier extension, defect, depth, Henselian field, Okutsu sequence, valuation, ramification ideals}

\begin{document}
\subjclass[2010]{13A18 (12J10)}

\begin{abstract}
For a valued field $(K,v)$, with a fixed extension of $v$ to the algebraic closure $\overline K$ of $K$, and an algebraic element $\theta\in\overline K$, we study the possible values of $\theta-\theta'$, where $\theta'$ runs through the $K$-conjugates of $\theta$. We compare the number of such values with the depth and the main invariant of $\theta$. Explicit defect examples show that the equalities familiar from the tame case may fail in either direction. We also relate these conjugate distances to ramification ideals of $(K(\theta)/K,v)$ and compute the ramification ideals in several defect extensions.
\end{abstract}

\maketitle

\section{Introduction}

Let $K$ be a field with a fixed algebraic closure $\overline K$, fix a valuation $v$ on $\overline K$ and set $\Gamma:=v\overline K$. For a separable element $\theta\in\overline K$ of degree greater than one, we are interested in the set
\[
S_\theta=\{v(\theta'-\theta)\mid \theta'\mbox{ is a }K\mbox{-conjugate of }\theta\mbox{ and }\theta\neq \theta'\}.
\]
The \textbf{Krasner's constant} $\omega(\theta)$ is, by definition, the maximum of $S_\theta$. It follows from the definition that $\#S_\theta=1$ if and only if
\begin{equation}\label{Kuhlque}
v(\theta-\theta')=\omega(\theta)\mbox{ for every conjugate }\theta'\mbox{ of }\theta,\theta'\neq \theta.
\end{equation}

We want to compare the cardinality of $S_\theta$ with natural invariants of $\theta$. The main invariant we have in mind is the \emph{depth} of $\theta$, which is defined as the number of possible degrees, different from $\deg_K(\theta)$, of elements appearing in an \emph{Okutsu sequence} for $\t$ (see Section \ref{secOS} for more details). This concept is not intrinsically associated to $\mathcal E:=(L/K,v)$ for $L=K(\theta)$. Different  generators of the same extension may have different depths. This leads to defining the depth of a simple algebraic extension of valued fields as
\[
\dep(\mathcal E):=\min\{\dep(\t)\mid L=K(\t)\}.
\]


We will say that $\theta$ \textbf{is pure} (or that \textbf{$\mathcal E$ is pure in $\theta$}) if $\dep(\theta)=1$.
\begin{Obs}
All the elements satisfying the definition of pure in \cite{CKR} and \cite{NS2026} satisfy the above condition.
\end{Obs}

The comparison undertaken here lies at the intersection of several classical
themes in valuation theory. Krasner's lemma made the distances between
conjugates a basic tool for controlling extensions of valued fields, and the
largest such distance is now commonly called Krasner's constant. The closely
related main invariant was developed as a measure of how well an algebraic
element can be approximated by elements of smaller degree; see, among others,
\cite{Kan5,Kan1,Kan2,Kan3,Kan4,OkS}. These approximation problems are naturally
encoded by key polynomials and Mac Lane--Vaqui\'e chains
\cite{Vaq,MLV,AFFGNR}. Okutsu sequences provide a finite form of the same
approximation data in the defectless Henselian setting, and their successive
degrees lead to the notion of depth used in this paper \cite{OS}.

The presence of defect changes this picture substantially. Kaplansky's theory
of pseudo-convergent sequences already showed that immediate algebraic
extensions require approximation phenomena with no best approximation
\cite{Kap}. In positive characteristic, Artin--Schreier extensions furnish the
basic degree-$p$ examples, and the distinction between dependent and
independent defect is reflected in their distance cuts \cite{Kuhl}. More
recently, these cuts have also appeared in descriptions of K\"ahler
differentials, differents, and ramification ideals
\cite{Tha1,Tha2,KR,CKR,NS2026}. Our purpose is to compare, in this defect
setting, three pieces of information that agree particularly well in the tame
case: conjugate distances, lower-degree approximation, and ramification
ideals.

For $\theta\in\overline K\setminus K$, set
\[
M_\theta:=\{v(\theta-b)\mid b\in\overline K\mbox{ and }\deg_K(b)<\deg_K(\theta)\}.
\]
Following \cite{Kan1}, whenever $M_\theta$ is bounded above, we define the \textbf{main invariant of $\theta$} by
\[
\delta(\theta):=\sup M_\theta,
\]
where the supremum is taken in the Dedekind order completion of $\Gamma$. If $M_\theta$ is unbounded above, we set $\delta(\theta)=+\infty$.
We distinguish this invariant from the cut
\[
\Delta(\theta):=M_\theta^+,
\]
whose left cut set is the smallest initial segment of $\Gamma$ containing $M_\theta$, namely
\[
\Delta(\theta)^L=\{\gamma\in\Gamma\mid\gamma\leq s\mbox{ for some }s\in M_\theta\}.
\]
Here we use the cut notation introduced in Section \ref{secOS}. If $\delta(\theta)=a\in\Gamma$, then $\Delta(\theta)=a^+$ when the supremum is attained, and $\Delta(\theta)=a^-$ otherwise. Thus $\Delta(\theta)$ also records whether the supremum is attained, whereas $\delta(\theta)$ does not. The invariants $d_m(\theta)$ defined below remain cut-valued.

The first goal of this paper is to discuss the following questions.
\begin{Que}\label{mainconj}
For a separable $\theta\in \overline{K}\setminus K$, is it true that
\[
\#S_\theta=\dep(\theta)?
\]
\end{Que}
A particular case of the above question is the following question, which was proposed by Franz-Viktor Kuhlmann.
\begin{Que}\label{kuhlmaconj}
Is it true that if a separable element $\theta\in\overline{K}\setminus K$ is pure, then \eqref{Kuhlque} is satisfied?
\end{Que}

For a field $L$, $K\subseteq L\subseteq \overline K$, denote by $\VR_L$ the valuation ring, by $vL$ the value group and by $Lv$ the residue field of $(L,v)$. Also, denote by $\mathcal M_L$ the maximal ideal of $\VR_L$. The following question was proposed in \cite{OkS}.

\begin{Que}\label{connart}
Take a separable element $\theta\in\overline{K}\setminus K$ and set $L=K(\theta)$. Is it true that if $Lv/Kv$ is separable and $\gcd((vL:vK),p)=1$, where $p$ is the characteristic exponent of $Kv$, then
\[
\delta(\theta)=\omega(\theta)?
\]
\end{Que}

These questions have been extensively studied in the \emph{defectless case} (see Section \ref{notation} for the definition and \cite{NN} for the main properties of the defect). For instance,  each of the works \cite{Kan1}, \cite{Kan2}, \cite{Kan3}, \cite{Kan4}, \cite{Kan5}, \cite{OkS} and \cite{OS}, to cite a few, treat some particular case of the questions above.

It is well-known that if $(K,v)$ is Henselian and $\theta$ is \emph{tame} over $K$, then all the above questions have an affirmative answer (see Section \ref{tamecase}). Tame extensions are, by definition, defectless. Hence, it is natural to ask what happens in the \emph{defect case}.
In the tame case, the supremum defining $\delta(\theta)$ is attained, so that $\delta(\theta)=\omega(\theta)$ and $\Delta(\theta)=\omega(\theta)^+$. Question \ref{connart} asks only for equality of the supremum and the Krasner constant, without requiring the supremum to be attained.

The first consequence of this paper is that in the defect case the situation is very different. More precisely, in Section \ref{mainandnart} we present an example of a field $K$ and separable elements $\beta,\theta\in \overline K\setminus K$ for which
\begin{equation}\label{firstexampl}
\#S_\theta=1<2=\dep(\theta)
\end{equation}
and $(K(\beta)/K,v)$ is \emph{immediate} (in particular, the conditions of Question \ref{connart} are satisfied) with
\[
\delta(\beta)<\omega(\beta).
\]

In view of \eqref{firstexampl}, one could ask whether it is true that $\#S_\theta\leq \dep(\theta)$ for every separable $\theta\in\overline K\setminus K$. Observe that if this is true for some $\theta$, then the answer to Question \ref{kuhlmaconj} is affirmative for such $\theta$. However, in Section \ref{monster} we present an example of a field $K$ and a separable element $\theta$ such that
\[
\#S_\theta=2>1=\dep(\theta).
\] 

A motivation to study the questions above comes from the study of \emph{ramification ideals}. These objects appear as an important part in some recent problems in ramification theory of valued fields (for example, see \cite{KR}, \cite{Topics}, \cite{Tha1} and \cite{Tha2}). For any nonempty subset $S$ of $vL$ bounded from below, we define the $\VR_L$-fractional ideal $I_S$ of $L$ by
\[
I_S=\{c\in L\mid vc\geq s\mbox{ for some }s\in S\},
\] 
and set $I_\emptyset=\{0\}$. Suppose that $\mathcal E:=(L/K,v)$ is a finite Galois extension and that $v$ is invariant under $\gal(L/K)$. For $\sigma\ne id$ we denote
\[
I_\sigma=\left\{c\in L\mid vc\geq v\left(\frac{\sigma b-b}{b}\right)\mbox{ for some }b\in L^*\right\},
\]
i.e., $I_\sigma=I_{S_\sigma}$ where
\[
S_\sigma=\left\{v\left(\frac{\sigma b-b}{b}\right)\mid b\in L^*,\ \sigma b\ne b\right\}.
\]
We set $I_{id}=\{0\}$.
For a subgroup $H$ of $\gal(L/K)$ we define
\[
I_H=\bigcup_{\sigma\in H} I_\sigma.
\]
This is an $\VR_L$-submodule of $L$, since the valuation ideals $I_\sigma$
are totally ordered by inclusion.

It might happen that for two distinct subgroups $H$ and $H'$ of $\gal(L/K)$ we have $I_H=I_{H'}$. We will consider the set of $\VR_L$-submodules
\[
{\rm Ram}(\mathcal E):=\{I_H\mid H\mbox{ is a subgroup of }\gal(L/K), H\neq \{id\}\}.
\]
We will refer to the elements of ${\rm Ram}(\mathcal E)$ which are contained in $\mathcal M_L$ as the \textbf{ramification ideals of }$\mathcal E$.

In \cite{KR} it is shown that if $\mathcal E$ is a Galois extension of degree and defect $p$ (where $p$ is a prime number), then
\[
{\rm Ram}(\mathcal E)=\{\mathcal I\},\mbox{ i.e., }\#{\rm Ram}(\mathcal E)=1.
\]
Moreover, for every generator $\theta$ of $L/K$ and every $\sigma\in \gal(L/K)\setminus\{id\}$ we have 
\begin{equation}\label{mainconjst}
\mathcal I=I_{v(\sigma \theta-\theta)-D_1(\theta,K)},\mbox{
where }D_1(\theta,K)=\{v(\theta-c)\mid c\in K\}.
\end{equation}
The above discussion tells us that if $\mathcal E$ is a Galois extension of degree and defect $p$, then for every generator $\theta$ of $L/K$ we have 
\[
{\rm Ram}(\mathcal E)=\{I_{s-D_1(\theta,K)}\mid s\in S_\theta\}.
\]
Hence, it is natural to ask the following question.

\begin{Que}\label{questuiabouraminf}
For a finite nontrivial Galois extension $\mathcal E=(L/K,v)$, does there exist a generator $\theta$ of $L/K$ for which
\[
{\rm Ram}(\mathcal E)=\{I_{s-D_1(\theta,K)}\mid s\in S_\theta\}?
\]
\end{Que}

In Section \ref{depthramifi} we obtain, among other results, an affirmative
answer to Question \ref{questuiabouraminf} for the immediate composita considered
in Corollary \ref{linealrydisoj}. Under the hypotheses of that corollary, its proof
produces a generator $\theta$ of $L/K$ satisfying
\[
\{I_{s-D_1(\theta,K)}\mid s\in S_\theta\}
=\{I_{-D_1(\theta,K)}\}=\{\mathcal M_L\}
={\rm Ram}(L/K,v).
\]

For the nonabelian tower $\mathcal E=(N/K,v)$ in characteristic $p>2$ considered in Section \ref{computationcassesAsd}, the ramification computation below and the depth calculation in \cite[Section 4.4]{NNASD} give
\[
\#{\rm Ram}(\mathcal E)=1<2=\dep(\mathcal E)
\]
(see Remark \ref{depn3famif1}).

Finally, Theorem \ref{generalizakuhlm} indicates that ${\rm Ram}(\mathcal E)$ is closely related to $\min S_\theta$. Namely, for a nontrivial subgroup $H$ of $\gal(L/K)$ we define
\[
S(\theta,H)=\{v(\sigma \theta-\theta)\mid \sigma\in H\setminus\{id\}\}.
\]
Observe that if $\mathcal E$ is a nontrivial Galois extension with Galois group $G$, then $S_\theta=S(\theta, G)$. Theorem \ref{generalizakuhlm} shows that for a \emph{unibranched immediate} Galois extension $(L/K,v)$ of rank one and a nontrivial subgroup $H$ of order smaller or equal to the \emph{characteristic exponent} of $(K,v)$ we have $I_H=I_S$ where
\[
S=\min S(\theta,H)-D_1(\theta,K_H)
\]
and $K_H$ is the fixed field of $H$. In Section \ref{depthramifi} we also present consequences of this result. In Section \ref{computationcassesAsd}, we compute the ramification ideals for the $p>2$ tower used here and separately describe the degeneration of the present equations in characteristic two.

For convenience, the main conclusions of the paper are summarized below.
The hypotheses displayed in the middle column are essential.
\begin{center}
\small
\renewcommand{\arraystretch}{1.18}
\begin{tabular}{@{}>{\raggedright\arraybackslash}p{3.1cm}>{\raggedright\arraybackslash}p{7.1cm}>{\raggedright\arraybackslash}p{3.0cm}@{}}
\hline
Setting & Main conclusion & Location \\
\hline
Tame benchmark & Over a Henselian base, $\#S_\theta=\dep(\theta)$ and $\delta(\theta)=\omega(\theta)$. & Section \ref{tamecase} \\
Dependent AS defect & In rank one, a dependent defect generator can satisfy $\delta(\beta)<\omega(\beta)$. & Section \ref{mainandnart} \\
Two disjoint AS extensions & Under the separated-cut and boundedness hypotheses, a primitive element satisfies $\#S_\theta=1<2=\dep(\theta)$. & Lemma \ref{congent1} \\
Balanced-cut example & An explicit Hahn-field construction gives $\#S_\theta=2>1=\dep(\theta)$. & Lemma \ref{dephtonfisone} \\
Ramification formula & For the stated immediate Galois extensions and subgroups of order at most $p$, $I_H$ is determined by $\min S(\theta,H)-D_1(\theta,K_H)$. & Theorem \ref{generalizakuhlm} \\
Independent AS compositum & If the Henselian compositum satisfying \textbf{(GE)} is immediate, then ${\rm Ram}(L/K,v)=\{\mathcal M_L\}$. & Corollary \ref{linealrydisoj} \\
Defect towers & For the nonabelian tower in characteristic $p>2$, $\#{\rm Ram}=1<2=\dep$; the characteristic-two equations instead yield depth one. & Proposition \ref{ramification-N} and Remark \ref{char-two-tower} \\
\hline
\end{tabular}
\end{center}

\medskip
\noindent\textbf{Acknowledgement.} This paper was inspired by discussions with Enric Nart, Franz-Viktor Kuhlmann and Mark Spivakovsky. In particular, the computations in Section \ref{computationcassesAsd} were provided by Nart. I specially thank Spivakovsky for finding and fixing a mistake in an earlier version of the example in Section \ref{monster}.

\section{Preliminaries}
\subsection{Notation and terminology}\label{notation}
Throughout this paper, we will denote by $\N$ the set of all positive and by $\N_0$ the set of all non-negative integers.
  
We will say that the extension $\mathcal E=(L/K,v)$ is \textbf{unibranched} if $v$ is the only extension of $v_{\mid_{K}}$ to $L$. For a field $L$, $K\subseteq L\subseteq \overline K$, we will denote by $L^h$ the \textbf{henselization} of $L$, i.e., the smallest subfield $L'$, $L\subseteq L'\subseteq \overline K$, such that $(\overline K/L',v)$ is unibranched. A valued field $(K,v)$ is \textbf{Henselian} if $K=K^h$, or equivalently, if $(\overline K/K,v)$ is unibranched. The \textbf{characteristic exponent of $(K,v)$} is defined as $1$ if ${\rm char}(Kv)=0$ and ${\rm char}(Kv)$ otherwise.

For a finite extension $\mathcal E=(L/K,v)$, we regard (in the natural way) $Kv$ as a subfield of $Lv$. The \textbf{defect} of $\mathcal E$ is the number
\[
d(\mathcal E)=\frac{[L^h:K^h]}{[Lv:Kv]\cdot (vL:vK)}.
\]
We will say that $\mathcal E$ is \textbf{defectless} if $d(\mathcal E)=1$. Otherwise, we will say that $\mathcal E$ is a \textbf{defect extension}. We will call $\mathcal E$ an \textbf{immediate extension} if
\[
vK=vL\mbox{ and }Lv=Kv.
\]

By an \textbf{Artin-Schreier (abbreviated by AS) extension} we mean a field extension $L/K$ of degree $p={\rm char}(K)>0$ which is generated by an element
\[
\alpha\in \overline K\setminus K\mbox{ such that }\alpha^p-\alpha\in K.
\]
In this case, we will say that $\alpha$ is an \textbf{AS element} over $K$. 

We will use Kuhlmann's characterization (see \cite{Kuhl}) of defect AS extensions as dependent or independent. In the rank-one cases considered below, $(K(\alpha)/K,v)$ is \textbf{independent} precisely when
\[
d_1(\alpha)=0^-
\]
as a cut in $\Gamma=v\overline K$; equivalently, $D_1(\alpha,K)$ is cofinal below zero. Otherwise the defect AS extension is \textbf{dependent}.

For a field $L$, $K\subseteq L\subseteq \overline K$, a subset $S$ of $vL$ will be called a \textbf{final segment} if for every $\gamma\in vL$ we have
\[
\gamma \geq s\in S\Lra \gamma\in S.
\]
We will also use the analogous notion of \textbf{initial segment}. For two subsets $S, S'$ of $vL$ we will consider the Minkowski sum $S+S'$, i.e., 
\[
S+S'=\{s+s'\mid s\in S\mbox{ and }s'\in S'\}.
\]
Also, we denote $-S'=\{-s'\mid s'\in S'\}$ and $S-S':=S+(-S')$. For $s\in vL$ and $S'\subseteq vL$ we denote $s-S'=\{s\}-S'$.

We will say  that the fields $L_1,\ldots,L_n$, $K\subseteq L_i\subseteq \overline K$, are linearly disjoint over $K$ if for every $i$, $1\leq i\leq r$, we have
\begin{equation}\label{linearlydisj}
L_i\mbox{ and }L_1\cdots L_{i-1}\cdot L_{i+1}\cdots L_n\mbox{ are linearly disjoint over }K.
\end{equation}
If $L_i/K$ is Galois, then \eqref{linearlydisj} is equivalent to
\[
L_i\cap\left(L_1\cdots L_{i-1}\cdot L_{i+1}\cdots L_n\right)=K.
\] 
We will say that  $\alpha_1,\ldots,\alpha_n\in\overline{K}$ are linearly disjoint over $K$ if $K(\alpha_1),\ldots,K(\alpha_n)$ are linearly disjoint over $K$.
\subsection{Okutsu sequences and depth}\label{secOS}
In this section, we do not make any assumption on the valued field $(K,v)$. Let us still denote by $v$ some fixed extension of $v$ to an algebraic closure $\kb$ of $K$. Denote $\g:=v\kb$.
From now on, we denote $\g\cup\{\infty\}$ simply by $\gi$.

Let $\t\in \overline K$ and $n\ge1$ be the degree of $\t$ over $K$. 
For every integer $1\le m\le n$, we define the \textbf{set of distances} of $\t$ to elements in $\kb$ of degree $m$ over $K$ as:
\[
D_m=D_m(\t,K):=\left\{v(\t-b)\mid b\in\kb,\ \deg_Kb=m\right\}\sub\gi.
\] 
Note that $\max(D_n)=\infty$.
The set $D_1(\theta,K)$ was studied by Blaszczok and Kuhlmann (see \cite{B,Kuhl}, for instance) for its connections with defect and immediate extensions. Since this maximal distance may not exist, we follow \cite{Kuhl} and we replace this concept with an analogous one in the context of cuts in $\g$.

A \textbf{cut} in $\g$ is a pair $\dta=(\dta^L,\dta^R)$ of subsets of $\g$ such that $$\dta^L< \dta^R\quad\mbox{ and }\quad \dta^L\cup \dta^R=\g.$$ 
The inequality $\dta^L< \dta^R$ means that $s<s'$ for all elements $s\in \dta^L$, $s'\in \dta^R$. Note that 
 $\dta^R=\g\setminus\dta^L$. Let us denote by $\op{Cuts}(\g)$ the set of all cuts in $\g$. 

For all $D\sub \g$ we denote by  $D^+$, $D^-$ the cuts determined by
\[
\dta^L=\{\ga\in \g\mid \exists s\in D: \ga\leq s\},\qquad \dta^L=\{\ga\in \g\mid \ga<D\},
\]
respectively.
Here $\gamma<D$ means that $\gamma<s$ for every $s\in D$.
If $D=\{\ga\}$, then we will write $\ga^+=(\g_{\le \ga},\g_{>\ga})$ instead of $\{\ga\}^+$ and $\ga^-=(\g_{<\ga}, \g_{\ge \ga})$ instead of $\{\ga\}^-$. These cuts are said to be 
\textbf{principal}.

The set $\op{Cuts}(\g)$  is totally ordered with respect to the following ordering:
\[
(\dta^L,\dta^R)\le (\ep^L,\ep^R)\ \sii\ \dta^L\sub \ep^L.
\]
The \textbf{improper} cuts 
$-\infty:=(\emptyset,\g)$, $\infty^-:=(\g,\emptyset)$ 
are the absolute minimal and maximal elements in $\op{Cuts}(\g)$, respectively.

\begin{Def}\label{distance}
For $m<n$, we define $d_m(\t)$ to be the cut $D_m^+\in \op{Cuts}(\g)$.  

We agree that $d_n(\t)$ is the improper cut $\infty^-$.
\end{Def}

We say that a nonempty subset $A\sub\kb$ has a \textbf{common degree} if all its elements have the same degree over $K$. In this case, we shall denote this common degree by $\deg_K A$. 

\begin{Def}\label{defOkS}
	An \textbf{Okutsu sequence} of $\t$ is a finite sequence 
\[
\left[A_0,A_1,\dots,A_{r-1},A_r=\{\t\}\right], 
\]
of common degree subsets of $\kb$ whose degrees grow strictly:
\begin{equation}\label{degOS}	
		1=m_0<m_1<\cdots<m_r=n,\qquad m_\ell=\deg_K A_\ell, \ \ 0\le\ell\le r,
\end{equation}
and	satisfy the following properties for all $\,0\le\ell< r$: \e
	
	(OS0) \ For all $b\in\kb$ such that $\deg_Kb<m_{\ell+1}$, we have $v(\t-b)\le  v(\t-a)$ for some $a\in A_\ell$.\e
	
	
	(OS1) \ $\#A_\ell=1$ whenever $\max\left(D_{m_{\ell}}\right)$ exists.\e
	
	(OS2) \ If $\max\left(D_{m_{\ell}}\right)$ does not exist, then we assume that $A_\ell$ is well-ordered with respect to the following ordering: \ $a<a'\ \sii \ v(\t-a)<v(\t-a')$.\e
	
	(OS3) \ For all $a\in A_{\ell}$, $b\in A_{\ell+1}$, we have $v(\t-a)<v(\t-b)$.
\end{Def}

Let us discuss the existence and construction of Okutsu sequences.
Consider the sequence of minimal degrees of the distances of $\t$:
\begin{equation}\label{degDist}	
1=q_0<q_1<\cdots<q_s=n,
\end{equation}
defined recursively as follows:

$\bullet$ \ $q_0=1$,

$\bullet$ \ as long as $q_{\ell-1}<n$, $q_\ell$ is the least integer $m>q_{\ell-1}$ such that there exists some $\ep\in D_m$ satisfying $\ep>D_{q_{\ell-1}}$.\e

Now, for each $0\le \ell<  s$ choose subsets $A_\ell\sub\kb$ of common degree $\deg_K A_\ell=q_\ell$ such that
\[
 \{v(\t-a)\mid a\in A_\ell\}\sub D_{q_\ell}
\]
is a well-ordered cofinal subset of $D_{q_\ell}$. Also, if for some $\ell$ there exists $\ga=\max\left(D_{q_\ell}\right)$, then we take $A_\ell=\{a\}$, for some $a\in\kb$ such that $\deg_K a=q_\ell$ and $v(\t-a)=\ga$. 

Finally, for $\ell>0$, we consider in $A_\ell$ only elements $a$ such that $v(\t-a)>D_{q_{\ell-1}}$. 

Clearly, $\left[A_0,A_1,\dots,A_{r-1},A_r=\{\t\}\right]$ is an Okutsu sequence of $\t$ and, conversely, all Okutsu sequences arise in this way. 

In particular, the sequence (\ref{degOS}) of degrees over $K$ of the sets $A_\ell$ is equal to the canonical sequence (\ref{degDist}) of minimal degrees of the distances of $\t$. That is, $r=s$ and $q_\ell=m_\ell$ for all $0\le \ell<r$.

\begin{Def}\label{OSdepth}
The length $r$ of any Okutsu sequence of $\t$ is called the depth of $\t$ and denoted by $\dep(\t)$. \end{Def}

The next result follows from \cite[Theorem 1.2]{NS2018} and Kaplansky's results on pseudo-convergent sequences with bounded approximation values in rank one \cite{Kap}.
\begin{Cor}\label{corkaplansky}
Assume that $(K,v)$ has rank one and ${\rm char}(K)=p>0$, and take $\theta\in\overline K\setminus K$. Suppose that $D_1(\theta,K)$ is bounded above in $\Gamma$ and has no maximum. Then the first degree $q_1>1$ in \eqref{degDist} is a power of $p$.

Moreover, if $q_1=p$, there exists $b\in\overline K$ with $[K(b):K]=p$ such that
\[
v(\theta-b)>\gamma\qquad\mbox{for every }\gamma\in D_1(\theta,K),
\]
and whose minimal polynomial over $K$ has one of the forms
\[
x^p-d\qquad\mbox{or}\qquad x^p-cx-d,
\]
where $d\in K$ and, in the second case, $c\in K^*$.
\end{Cor}
\subsection{The tame case}\label{tamecase}
The next result appears in \cite[Theorem 4.5]{OS} and \cite[Theorem 3.4]{OkS} and it shows that tame elements give affirmative answers to Questions \ref{mainconj} and \ref{kuhlmaconj}. Similar results can be deduced from \cite{Kan1}, \cite{Kan2}, \cite{Kan3}, \cite{Kan4} and \cite{DuKuh}.

Suppose that $\theta\in\overline{K}\setminus K$ is separable, defectless and unibranched over $K$. In this case, $\theta$ admits an Okutsu sequence
\[
[\alpha_0,\alpha_1,\ldots, \alpha_r=\theta]
\]
of degrees $1=m_0\mid m_1\mid\ldots\mid m_{r-1}\mid m_r=n:=\deg_K(\theta)$ for $\theta$. Denote 
\[
\delta_{-1}=-\infty<\delta_0:=v(\theta-\alpha_0)<\ldots< \delta_{r-1}:=v(\theta-\alpha_{r-1})<\delta_r=\infty.
\]
We denote by $\{\delta_0^{t_0},\ldots, \delta_{r-1}^{t_{r-1}}\}$ the multiset whose underlying set is $\{\delta_0,\ldots,\delta_{r-1}\}$ and each $\delta_i$ appears with multiplicity $t_i$.

A field extension $(L/K,v)$ is \textbf{tame} if it is unibranched and defectless, the residue extension $Lv/Kv$ is separable, and $\gcd((vL:vK),p)=1$, where $p$ is the characteristic exponent of $Kv$. For $\theta\in\overline K$ we say that $\theta$ is tame if $(K(\theta)/K,v)$ is tame.

\begin{Teo}\cite[Theorem 4.5]{OS}
Assume that $(K,v)$ is Henselian. If $\theta$ is tame over $K$, then the multiset of values $v(\theta'-\theta)$, where $\theta'$ ranges through the conjugates different from $\theta$, is
\[
\{\delta_0^{t_0},\ldots, \delta_{r-1}^{t_{r-1}}\},
\]
where $t_i=(n/m_i)-(n/m_{i+1})$.
\end{Teo}

\begin{Cor}
Assume that $(K,v)$ is Henselian. If $\theta$ is tame over $K$, then the answer to Questions \ref{mainconj}, \ref{kuhlmaconj} and \ref{connart} is affirmative. 
\end{Cor}

\section{The defect case}
The situation in the defect case is very different. Let $p$ be the characteristic exponent of $(K,v)$. The following defect cases are known examples when the answer to Question \ref{mainconj} is affirmative:
\begin{enumerate}
\item If $(K(\theta)/K,v)$ is a Galois defect extension of degree $p$, then
\[
\# S_\theta=\dep(\theta)=1.
\]
For details, see \cite{KR}.
\item Examples $D1$ and $D2$ of \cite{OS} are situations where $L=K(\theta)$ and
\[
p=d(L/K,v), (L/K,v)\mbox{ is not immediate  and }\#S_\theta=2=\dep(\theta).
\] 
\item Example E of \cite{OS} is a situation where $L=K(\theta)$ and
\[
p^2=[L:K]=d(L/K,v)\mbox{ and }\#S_\theta=2=\dep(\theta).
\]
\end{enumerate}
 
In the next sections we present examples of defect extensions for which the answer to Questions \ref{mainconj}, \ref{kuhlmaconj} and \ref{connart} is negative.

\subsection{Counterexamples for Questions \ref{mainconj} and \ref{connart}}\label{mainandnart}
Let $(K,v)$ be a valued field admitting a dependent defect AS extension $(L/K,v)$ with AS generator $\beta$. Assume moreover that $v$ has rank one and, specifically, that
\[
d_1(\beta)=\delta^{-}\mbox{ for }\delta\in \R\mbox{ and }\delta<0.
\]
Here we fix an order embedding of the divisible hull of $vK$ into $\R$, and
$\delta^-$ denotes the cut with left set $\{\gamma\in\Gamma\mid\gamma<\delta\}$;
the endpoint $\delta$ need not belong to $\Gamma$.

\begin{Obs}
Since $D_1(\beta,K)$ is bounded above and has no maximum, Corollary \ref{corkaplansky} shows that the first degree of an approximation improving on all values in $D_1(\beta,K)$ is a nontrivial power of $p$. As $\deg_K(\beta)=p$, this first degree is $p$. Hence no element of degree strictly smaller than $p$ improves on $D_1(\beta,K)$, and therefore
\[
\Delta(\beta)=d_1(\beta)=\delta^-<0^-=\omega(\beta)^-.
\]
Consequently, in terms of the main invariant defined as a supremum, we have
\[
\delta(\beta)=\delta<0=\omega(\beta).
\]
In particular, the answer for Question \ref{connart} for $\beta$ is negative.
\end{Obs}

The next lemma provides an easy answer to Question \ref{mainconj}.
\begin{Lema}\label{congent1}
Assume that $(K,v)$ has rank one and characteristic $p>0$. Take defect AS elements $\alpha,\beta\in\overline{K}$ such that $K(\alpha)$ and $K(\beta)$ are linearly disjoint over $K$. Suppose that there exist $a,b\in K$ such that
\begin{equation}\label{congent}
v(ac+bd)=0\mbox{ for every }(c,d)\in\F_p^2\setminus \{(0,0)\}.
\end{equation}
Assume also that $D_1(b\beta,K)$ is bounded above and has no maximum.
If $d_1(b\beta)<d_1(a\alpha)$, then $\theta=a\alpha+b\beta$ is a generator of $K(\alpha,\beta)$ such that
\[
\#S_\theta=1<2=\dep(\theta).
\]
\end{Lema}
\begin{proof}
Since $K(\alpha)$ and $K(\beta)$ are linearly disjoint over $K$, by \cite[Lemma 2.3]{NNASD}, all the conjugates of $\theta$ over $K$ are
\[
\theta+a\F_p+b\F_p.
\]
Hence, condition \eqref{congent} guarantees that $\theta$ is a generator of $K(\alpha,\beta)$ and $\#S_\theta=1$.

Since $d_1(b\beta)<d_1(a\alpha)$, the equality
$d_1(\theta)=d_1(b\beta)$ holds, and we can take an element $l\in K$ such that
\[
d_1(b\beta)<v(a\alpha-l)
\]
and consider $\eta=b\beta+l$. Then
\[
v(\theta-\eta)>d_1(b\beta)=d_1(\theta)
\]
and $p=\deg_K(\eta)<\deg_K(\theta)=p^2$. Thus the first nontrivial degree $q_1$ in \eqref{degDist} satisfies $q_1\leq p$. Corollary \ref{corkaplansky}, applied to $\theta$, shows that $q_1$ is a nontrivial power of $p$, so $q_1=p$. Hence the degree sequence is $1<p<p^2$ and $\dep(\theta)=2$.
\end{proof}

\subsubsection{Concrete example}\label{exampleeasy} According to Lemma \ref{congent1}, it is enough to build a field $K$ with a rank one valuation $v$ and defect AS elements $\alpha,\beta\in\overline K$, linearly disjoint over $K$, such that $\alpha$ is independent ($d_1(\alpha)=0^-$) and $\beta$ is dependent ($d_1(\beta)=\delta^-$, $\delta<0$).

We will build this example based on \cite[Example 4.19]{Kuhl}. 
Set $\Gamma=(\Z+\pi \Z)\otimes_{\Z}\Q$ and $k=\overline{\F_p}$. Our example will be constructed for a field $K$ such that
\[
K_0=k(t,t^{\pi})\subset K\subset \mathbb H=k\left(\left(t^\Gamma\right)\right)
\]
equipped with the $t$-adic valuation $v$. The Artin-Schreier operator $AS:\overline K\lra \overline K$ is defined by $AS(a)=a^p-a$.

Let
\[
a_1=t^{-\frac{1}{p}}+\ldots+t^{-\frac{1}{p^n}}+\ldots\in \mathbb H,
\]
i.e., $AS(a_1)=t^{-1}$. Iteratively, we construct
\[
a_{\ell+1}=-a_\ell^{\frac{1}{p}}-\ldots-a_\ell^{\frac{1}{p^n}}-\ldots\in \mathbb H,
\]
i.e., $AS(a_{\ell+1})=-a_\ell$. Also, set $b_\ell=t^{\frac{\pi}{p^\ell}}$. Define
\[
K_\ell=K_0(a_\ell,b_\ell)\mbox{ and }K=\bigcup_{\ell\in\N}K_\ell.
\]
One can easily see that
\begin{equation}\label{equakl}
vK_\ell=\frac{1}{p^\ell}\left(\Z+\pi\Z\right)\mbox{ for every }\ell\in\N.
\end{equation}
Let $ \beta,\alpha\in \mathbb H$ be given by
\[
\beta=t^{-\frac{p+1}{p}}+\ldots+t^{-\frac{p+1}{p^\ell}}+\ldots
\mbox{ and }\alpha=t^{-\frac{\pi}{p}}+\ldots+t^{-\frac{\pi}{p^\ell}}+\ldots,
\]
i.e.,
\[
AS(\beta)=t^{-p-1}\mbox{ and }AS(\alpha)=t^{-\pi}.
\]
If we set
\[
c_{\ell}=a_1+\ldots+a_{\ell}\mbox{ and }d_\ell=b_1^{-1}+\ldots+b_\ell^{-1},
\]
then
\[
t^{-1}=a_1^p-a_1=a_1^p+a_2^p-a_2=\ldots=a_1^p+\ldots+a_\ell^p-a_\ell=c_\ell^p-a_{\ell}.
\]
Consequently,
\[
v(t^{-\frac{1}{p}}-c_{\ell})=v\left(a_\ell^{\frac{1}{p}}\right)=-\frac{1}{p^{\ell+1}}.
\]
Then
\begin{equation}\label{psicusbeta}
v(\beta-t^{-1}c_\ell)=-1-\frac{1}{p^{\ell+1}}.
\end{equation}
Also, it is easy to see that
\begin{equation}\label{psicusalpha}
v(\alpha-d_\ell)=-\frac{\pi}{p^{\ell+1}}.
\end{equation}

By \eqref{psicusbeta} and  \eqref{psicusalpha} we deduce that $d_1(\beta)\geq -1^-$ and $d_1(\alpha)\geq 0^-$. On the other hand, if there were $b\in K$ such that $v(\beta-b)\geq -1$ or $v(\alpha-b)\geq 0$, then we would have
\[
v(b-t^{-1}c_\ell)=v(\beta-t^{-1}c_\ell-(\beta-b))=-1-\frac{1}{p^{\ell+1}} \mbox{ for every }\ell\in\N
\]
or
\[
v(b-d_\ell)=v(\alpha-d_\ell-(\alpha-b))=-\frac{\pi}{p^{\ell+1}} \mbox{ for every }\ell\in\N.
\]
On the other hand, there exists $\ell\in\N$ such that $b\in K_\ell$. Since $c_\ell,d_\ell,t^{-1}\in K_\ell$ this implies that $b-t^{-1}c_\ell\in K_\ell$ or $b-d_\ell\in K_\ell$. This is a contradiction to \eqref{equakl}. Therefore, $d_1(\beta)=-1^{-}$ and $d_1(\alpha)=0^-$. In particular, $\alpha,\beta\notin K$, so both AS extensions have degree $p$. Moreover, $vK=\bigcup_\ell p^{-\ell}(\Z+\pi\Z)$ is $p$-divisible and $Kv=k$ is algebraically closed; hence both extensions have ramification index and residue degree one. Their bounded approximation cuts exclude a root in the henselization (in rank one, $K$ is dense in $K^h$), so the chosen prolongations are unibranched and both extensions have defect $p$.

By \cite[Lemma 2.3]{NNASD}, $\alpha,\beta\in \overline{K}$ are linearly disjoint over $K$. Take $b\in k\setminus\F_p$ and set $\theta=\alpha+b\beta$. By Lemma \ref{congent1} we obtain that
\[
\#S_\theta=1<2=\dep(\theta).
\]

\subsection{Counterexample for Question \ref{kuhlmaconj}}\label{monster}
Take defect AS elements $\alpha,\beta\in \overline{K}$ such that $K(\alpha)$ and $K(\beta)$ are linearly disjoint over $K$ and take $b\in K\setminus\F_p$. As before, $\theta=\alpha+b\beta$ is a generator of $K(\alpha,\beta)$.

If $d_1(\alpha)\neq d_1(\beta)+vb$, a degree-$p$ element approximates $\theta$ beyond $D_1(\theta,K)$, as in the proof of Lemma \ref{congent1}. We are interested in studying the case where $d_1(\alpha)= d_1(\beta)+vb$.
\begin{Obs}
Observe that in this case, if $d_1(\alpha)\neq d_1(\beta)$ (and hence $vb\neq 0$), then we have $\#S_\theta =2$ (because all the conjugates of $\theta$ are $\theta+\F_p+b\F_p$). In particular, if we find such situation with $\dep(\theta)=1$ we would have a counterexample to Question \ref{kuhlmaconj}. 
\end{Obs}

In what follows, we will build an example of a field $K$ and AS elements $\alpha,\beta\in\overline K$, linearly disjoint over $K$, such that $K(\alpha,\beta)=K(\theta)$ and
\[
\#S_\theta=2>1=\dep(\theta).
\]

Our example will be constructed for a field $K$ such that
\[
K_0=k(t,t^{\pi})\subset K\subset \mathbb H=k((t^\R))
\]
for $k=\overline {\F_p}$ equipped with the $t$-adic valuation $v$. For $b\in \mathbb H$ write
\[
b=\sum_{\gamma\in\Gamma}c_\gamma t^\gamma\in \mathbb H.
\]
For $\delta\in \R$ we will denote the \textbf{truncation of $b$ at $\delta$} by ${\rm trn}_\delta(b)$:
\[
{\rm trn}_\delta(b)=\sum_{\gamma<\delta} c_\gamma t^\gamma.
\]

Take a strictly increasing sequence $\{r_i\}_{i\in \N}$ of positive real numbers, with $r_1=p$ and $r_i\longrightarrow+\infty$, such that
\begin{equation}\label{ratindoenset}
\{\pi\}\cup\{1/r_i\}_{i\in \N}\mbox{ is a rationally independent set.}
\end{equation}
Such a sequence can be chosen recursively. Starting with $s_1=1/p$, choose
\[
0<s_{i+1}<\min\{s_i,1/(i+1)\},\qquad
s_{i+1}\notin\operatorname{span}_{\Q}\{\pi,s_1,\ldots,s_i\},
\]
and set $r_i=1/s_i$. At each step the excluded span is countable, whereas the indicated interval is uncountable.
For each $\ell\in \N$ set
\[
a_\ell=t^{-\frac{\pi}{p^\ell}}, b_\ell=t^{-\frac{1}{r_{\ell}}}-t^{-\frac{1}{r_{\ell+1}}}\in\mathbb H.
\]
Let
\[
K_\ell=K_0(a_1,\ldots,a_\ell,b_1,\ldots,b_\ell)\mbox{ and }K=K_0(a_\ell,b_\ell\mid \ell\in\N)=\bigcup_{\ell\in\N} K_\ell.
\]

\begin{Lema}
For every $\ell\in\N$ we have
\begin{equation}\label{equatsibestigl}
vK_\ell=G_\ell:=\frac{\pi}{p^\ell}\Z+\frac{1}{r_1}\Z+\ldots+\frac{1}{r_\ell}\Z+\frac{p}{r_{\ell+1}}\Z.
\end{equation}
In particular,
\begin{equation}\label{equianfgmarl}
\frac{\pi}{p^{\ell+1}},\frac1{r_{\ell+1}}\notin vK_\ell.
\end{equation}
\end{Lema}
\begin{proof}
For every $\ell\in\N$ write
\[
d_\ell=b_1+\ldots+b_\ell=t^{-\frac{1}{p}}-t^{-\frac{1}{r_{\ell+1}}}\mbox{ and }\tilde{t}=t^{-1}-d_\ell^{p}=t^{-\frac{p}{r_{\ell+1}}}.
\]
Then we have
\[
v(a_\ell)=-\frac{\pi}{p^\ell}, v(b_j)=-\frac{1}{r_j}\ \forall j, 1\leq j\leq \ell,\mbox{ and }v\left(\tilde t\right)=-\frac{p}{r_{\ell+1}}.
\]
Hence, $G_\ell\subseteq vK_\ell$.

For the opposite inclusion, take any element $b\in K_\ell$. Then $b=f/g$ for
\[
f,g\in R:=k[t^{-1},t^{-\pi},a_1,\ldots,a_\ell, b_1,b_2,\ldots,b_\ell].
\]
We will show that for every nonzero $f\in R$ we have $v(f)\in G_\ell$ and \eqref{equatsibestigl} will follow.
For a given monomial
\[
M=ct^{-s}t^{-m\pi}a_1^{e_1}\ldots a_\ell^{e_\ell}b_1^{s_1}b_2^{s_2}\ldots b_\ell^{s_\ell}\in R
\]
we replace $t^{-1}=\tilde{t}+d_\ell^p$ and $b_1=d_\ell-b_2-\ldots-b_\ell$. Then every element $f\in R$ can be written as
\begin{equation}\label{equandkformo}
f=\sum_{i=1}^n M_i
\end{equation}
where each $M_i$ is of the form
\begin{equation}\label{eqagpfroj}
M=c\tilde t^{s} t^{-\frac{u}{p^\ell}\pi}d_\ell^{s_1}b_2^{s_2}\ldots b_\ell^{s_\ell}\mbox{ for some }
c\in k, s,u,s_1,\ldots,s_\ell\in\N_0.
\end{equation}
Since the value of such $M$ is
\[
-s\frac{p}{r_{\ell+1}}-u\frac{\pi}{p^\ell}-\frac{s_1}{p}-\frac{s_2}{r_2}-\ldots-\frac{s_\ell}{r_\ell}\in G_\ell
\]
and $\{\pi,1/p,1/r_2,\ldots,1/r_{\ell+1}\}$ is rationally independent, different exponent tuples give different values. After collecting monomials with the same exponent tuple and removing zero coefficients in \eqref{equandkformo}, all remaining monomials therefore have distinct values. In particular,
\[
v(f)=\min_{1\leq i\leq n}\left\{v\left(M_i\right)\right\}\in G_\ell.
\]
This concludes the proof of \eqref{equatsibestigl}.

Finally, membership of $\pi/p^{\ell+1}$ in $G_\ell$ would, by rational independence, require an integer $u$ satisfying $u/p^\ell=1/p^{\ell+1}$, hence $u=1/p$. Likewise, membership of $1/r_{\ell+1}$ would require an integer $s$ with $ps=1$. Both are impossible, proving \eqref{equianfgmarl}.
 \end{proof}

\begin{Def}
We say that an element $a\in \mathbb H$ has \textbf{no finite limits} if for every $\delta\in \R$ we have
\[
\#{\rm supp}({\rm trn}_\delta(a))<\infty.
\]
A subfield $L$ of $\mathbb H$ has no finite limits if every element $a\in L$ has no finite limits.
\end{Def}
\begin{Obs}
The field $K$ above has no finite limits.
\end{Obs}
\begin{proof}
Every element of $K$ is a quotient of two finite Hahn sums, since each generator of every $K_\ell$ has finite support. For a nonzero finite sum $g$, write $g=at^\lambda(1+h)$ with $a\in k^*$ and either $h=0$ or $h$ a finite sum with $v(h)>0$. In the latter case,
\[
g^{-1}=a^{-1}t^{-\lambda}\sum_{n\geq0}(-h)^n.
\]
For any real bound, only finitely many $n$ contribute below that bound, because $v(h^n)=nv(h)\longrightarrow+\infty$; each $h^n$ has finite support. Multiplication by a finite sum preserves this property. Thus every element of $K$ has no finite limits.
\end{proof}

Let $\alpha,\beta\in \mathbb H$ be given by
\[
\alpha=t^{-\frac{\pi}{p}}+\ldots+t^{-\frac{\pi}{p^\ell}}+\ldots
\mbox{ and }\beta=t^{-\frac{p+1}{p}}+\ldots+t^{-\frac{p+1}{p^\ell}}+\ldots,
\]
i.e.,
\[
AS(\alpha)=t^{-\pi}\mbox{ and }AS(\beta)=t^{-p-1}.
\]

\begin{Lema}
We have
\begin{equation}\label{impodist}
d_1(\alpha)=d_1\left(t^{-\frac{1}{p}}\right)=0^-\mbox{ and }d_1(\beta)=-1^-.
\end{equation}
\end{Lema}
\begin{proof}
If we set
\[
c_\ell=a_1+\ldots+a_{\ell}\mbox{ and }d_{\ell}=b_1+\ldots+b_\ell,
\]
for every $\ell\in\N$, $\ell>1$, then
\[
v(\alpha-c_\ell)=-\frac{\pi}{p^{\ell+1}} \mbox{ and }v(t^{-\frac{1}{p}}-d_\ell)=-\frac{1}{r_{\ell+1}}.
\]
Both displayed values tend to zero from below, using $r_\ell\longrightarrow+\infty$. Hence, $d_1(\alpha)\geq 0^-$ and $d_1\left(t^{-\frac{1}{p}}\right)\geq 0^-$.

Suppose, aiming for a contradiction, that $d_1(\alpha)>0^-$ or $d_1\left(t^{-\frac{1}{p}}\right)> 0^-$. Then there would exist $b\in K$ such that
\[
v(b-\alpha)\geq 0\mbox{ or }v(b-t^{-\frac 1p})\geq 0.
\]
Taking $\ell\in\N$ such that $b\in K_\ell$ we would obtain that
\[
v(b-c_\ell)=-\frac{\pi}{p^{\ell+1}}\mbox{ or }v(b-d_\ell)=-\frac{1}{r_{\ell+1}}.
\]
This is a contradiction to \eqref{equianfgmarl}.
It remains to compute the distance of $\beta$. Put $z=t^{-1/p}$. Then
\[
v(\beta-t^{-1}z)=-\frac{p+1}{p^2}>-1.
\]
Since $v(z-d_\ell)=-1/r_{\ell+1}$, we obtain
\[
v(\beta-t^{-1}d_\ell)=-1-\frac1{r_{\ell+1}}\longrightarrow-1.
\]
Conversely, if $b\in K$ satisfied $v(\beta-b)\geq-1$, then
\[
v(z-tb)\geq\min\{v(z-t\beta),v(t\beta-tb)\}\geq0,
\]
contradicting $d_1(z)=0^-$. Thus $d_1(\beta)=-1^-$.
\end{proof}

Set $\theta=\alpha+t\beta$ so that
\[
{\rm trn}_0(\theta)=t^{-\frac{1}{p}}+t^{-\frac{\pi}{p}}+\ldots+t^{-\frac{\pi}{p^\ell}}+\ldots
\]
The computed cuts show that $\alpha,\beta\notin K$, so both AS extensions have degree $p$. If these extensions were equal, then $\beta=i\alpha+a$ for some $i\in\F_p^*$ and $a\in K$, which would give $D_1(\beta,K)=D_1(\alpha,K)$, contrary to \eqref{impodist}. Hence they are linearly disjoint. Their compositum has Galois group $\F_p^2$, acting by independent translations of $\alpha$ and $\beta$. The conjugates of $\theta$ are therefore
\[
\theta+t\F_p+\F_p.
\]
These $p^2$ conjugates are distinct, since $i+tj=0$ with $i,j\in\F_p$ forces $i=j=0$. Hence $[K(\theta):K]=p^2$ and $\theta$ generates $K(\alpha,\beta)$. Moreover,
\[
v(i+tj)=
\begin{cases}
0,&i\ne0,\\
1,&i=0,\ j\ne0,
\end{cases}
\qquad\text{so that } S_\theta=\{0,1\}.
\]
\begin{Lema}\label{dephtonfisone}
We have $\dep(\theta)=1$.
\end{Lema}
\begin{proof}
Put $z=t^{-1/p}$ and $w=t\beta-z$. Then
\[
v(w)=1-\frac{p+1}{p^2}>0,\qquad \theta=z+\alpha+w.
\]
For $c_\ell=\sum_{j=1}^{\ell}a_j$ and $d_\ell=\sum_{j=1}^{\ell}b_j$, rational independence ensures that the two negative values below are distinct. Hence
\[
v\bigl(\theta-(c_\ell+d_\ell)\bigr)
=\min\left\{-\frac{\pi}{p^{\ell+1}},-\frac1{r_{\ell+1}}\right\}\longrightarrow0.
\]
On the other hand, ${\rm trn}_0(\theta)=z+\alpha$ has infinite support, whereas ${\rm trn}_0(a)$ has finite support for every $a\in K$. Thus $v(\theta-a)<0$ for all $a\in K$, and $d_1(\theta)=0^-$. In particular, $D_1(\theta,K)$ is bounded above and has no maximum.

Let $m$ be the least degree over $K$ of an element $\epsilon$ such that $v(\theta-\epsilon)>\gamma$ for every $\gamma\in D_1(\theta,K)$. By cofinality below zero, this is equivalent to $v(\theta-\epsilon)\geq0$. We have $1<m\leq p^2$, since $\theta$ itself is admissible. The valuation of $K$ has rank one, as $vK$ is a nonzero subgroup of $\R$. Corollary \ref{corkaplansky} therefore shows that $m$ is a power of $p$. If $m<p^2$, then $m=p$, and that corollary supplies such an $\epsilon$ with minimal polynomial
\[
x^p-d\qquad\text{or}\qquad x^p-cx-d,
\]
where $d\in K$ and $c\in K^*$ in the second case. We exclude both possibilities explicitly.

In a common valued overfield containing $\mathbb H$ and $\epsilon$, write
\[
\epsilon=z+\alpha+u,\qquad v(u)\geq0.
\]
This follows from $v(\theta-\epsilon)\geq0$ and $v(w)>0$; it does not require $\epsilon$ to belong to $\mathbb H$.

If $\epsilon^p=d$, then
\[
d=t^{-1}+\alpha^p+u^p,
\qquad \alpha^p=t^{-\pi}+\alpha,\qquad v(u^p)\geq0.
\]
Therefore ${\rm trn}_0(d)=t^{-1}+t^{-\pi}+\alpha$ has infinite support, contradicting the fact that $K$ has no finite limits.

Suppose instead that $\epsilon^p-c\epsilon=d$. Then
\begin{equation}\label{explicit-depth-identity}
d=t^{-1}+t^{-\pi}+(1-c)\alpha-cz+(u^p-cu).
\end{equation}
We show first that $v(c)=0$ and that its coefficient at exponent zero is $1$.

Write $q=v(c)$ and let $\lambda\in k^*$ be the coefficient of $t^q$ in $c$. If $q<0$, then $v(u^p-cu)\geq q$. The terms $t^{-1}$, $t^{-\pi}$, $\alpha$ and $cz$ have only finitely many support exponents below $q$. By contrast, for every sufficiently large $j$, the coefficient of $c\alpha$ at $q-\pi/p^j$ equals $\lambda$. Indeed, only finitely many exponents $s$ of $c$ satisfy $s<q+\pi/p$, and for each such $s>q$ only finitely many terms $s-\pi/p^j$ lie below $q$. Thus the infinite tail contributed by $\lambda t^q\alpha$ cannot cancel. Comparing the Hahn-field terms in \eqref{explicit-depth-identity} below $q$ forces $d$ to have infinite support below $q$, a contradiction.

If $q>0$, the negative support of $c\alpha$ is finite: only finitely many exponents $s$ of $c$ lie below $\pi/p$, and for each such $s>0$ only finitely many $j$ satisfy $s-\pi/p^j<0$. The negative support of $cz$ is also finite, and $v(u^p-cu)\geq0$. The infinite negative tail of $\alpha$ in \eqref{explicit-depth-identity} therefore cannot cancel, again a contradiction. Hence $q=0$.

Write $c=\lambda+h$, where either $h=0$ or $v(h)>0$. The preceding argument shows that $h\alpha$ has finite negative support. Consequently, for every sufficiently large $j$, comparison at exponent $-\pi/p^j$ in \eqref{explicit-depth-identity} gives coefficient $1-\lambda$. Since $d$ has finite negative support, $\lambda=1$. Thus $v(c)=0$ and either $c=1$ or $v(1-c)>0$.

If $c=1$, put $A=0$. Otherwise, put $q'=v(1-c)>0$ and choose $N\geq1$ with $\pi/p^{N+1}\leq q'$. Set
\[
A=\sum_{j=1}^{N}t^{-\pi/p^j}\in K.
\]
In both cases $v((1-c)(\alpha-A))\geq0$. Equation \eqref{explicit-depth-identity}, together with $v(c)=0$, now gives
\[
v\bigl(d-t^{-1}-t^{-\pi}-(1-c)A+cz\bigr)\geq0.
\]
The element
\[
a=\frac{t^{-1}+t^{-\pi}+(1-c)A-d}{c}\in K
\]
therefore satisfies $v(z-a)\geq0$, contradicting $d_1(z)=0^-$.

We have excluded $m=p$, so $m=p^2$. Equivalently, every $\epsilon$ with $v(\theta-\epsilon)\geq0$ has degree at least $p^2$ over $K$. The degree sequence \eqref{degDist} for $\theta$ is therefore $1<p^2$, and $\dep(\theta)=1$.
\end{proof}

\section{Ramification ideals}\label{depthramifi}

The main goal of this section is to understand ramification ideals in terms of the results in this paper. In Section \ref{beglasldjj} we present a generalization of \cite[Theorem 3.4]{KR}. In Section \ref{sectioncomposiu} we study immediate composita of defect AS extensions. In Section \ref{computationcassesAsd} we compute the ramification ideals for the specific $p>2$ tower used here and for its characteristic-two degeneration; in each extension explicitly treated there, the only ramification ideal is the maximal ideal.

\subsection{A result about ramification ideals}\label{beglasldjj}
For this section, suppose that $\mathcal E=(L/K,v)$ is a finite Galois extension of valued fields, with $L=K(\theta)$, and that the chosen valuation $v$ is invariant under $\gal(L/K)$. Thus the ideals $I_\sigma$, $I_H$, and the set ${\rm Ram}(\mathcal E)$ are understood with the convention of the introduction throughout this section. Take a nontrivial subgroup $H$ of $\gal(L/K)$. Recall that
\[
S(\theta,H)=\{v(\sigma \theta-\theta)\mid \sigma\in H\setminus\{id\}\}.
\]
and $K_H$ is the fixed field of $H$, i.e.,
\[
K_H=\{b\in L\mid \sigma b=b\mbox{ for every }\sigma\in H\}.
\]
Whenever $D_1(\theta,K_H)$ is bounded above, consider the bounded-below subset of $vL$ defined by
\[
S_H:=\min S(\theta,H)-D_1(\theta,K_H).
\]

The main result of this section is Theorem \ref{generalizakuhlm} below. This result is motivated by \cite{KR} and its proof is a variation of the results in that work. Before we state and prove it, we present Lemma \ref{lemakapland} that follows easily from a classic result of Kaplansky (see \cite[Lemma 8 and Lemma 10]{Kap}).

Let $p$ be the characteristic exponent of $(K,v)$. For all $s\in \N_0$, the \textbf{$s$-th Hasse-Schmidt derivative} $\partial_s$ on $\kx$ is defined by:
$$
f(x+y)=\sum_{0\le s}(\ps f)y^s \quad\mbox{for all } \,f\in\kx,
$$
where $y$ is another indeterminate. 

\begin{Lema}\label{lemakapland}
Assume that $(K,v)$ has rank one and characteristic $p>0$, that $\theta\in\overline K\setminus K$, and that $D_1(\theta,K)$ is bounded above in $\Gamma$ and has no maximum. Let $f\in K[x]$ be nonconstant, of degree $d\geq1$, and put $\beta_i=v(\partial_i f(\theta))\in\Gamma\cup\{\infty\}$ for $1\leq i\leq d$. Suppose that there exists $\gamma\in D_1(\theta,K)$ such that for every $c\in K$ we have
\[
v(\theta-c)\geq \gamma\Lra \beta_i=v\left(\partial_if(\theta)\right)=v\left(\partial_if(c)\right)\mbox{ for every }i, 1\leq i\leq d.
\]
Then there exist $h\in\{1,\ldots,d\}$, which is a power of $p$ and satisfies $\beta_h<\infty$, and $\gamma'\in D_1(\theta,K)$, with $\gamma'\geq\gamma$, such that
\[
\beta_h+hv(\theta-c)<\beta_i+iv(\theta-c)
\]
for every $c\in K$ with $v(\theta-c)\geq\gamma'$ and every $i\in\{1,\ldots,d\}\setminus\{h\}$. Terms with $\beta_i=\infty$ are interpreted in the usual extended-value sense. Moreover, if $1\leq d<p$, the derivative-stability hypothesis is automatic, $h=1$, and, after increasing $\gamma'$ within $D_1(\theta,K)$ if necessary, every such $c$ satisfies
\[
v(f(\theta))=v(f(c))<\beta_1+v(\theta-c)
=v(f(\theta)-f(c)).
\]
\end{Lema}

\begin{Obs}
Under the hypotheses of Lemma \ref{lemakapland}, every nonconstant polynomial of degree less than $p$ has dominant index $h=1$. Constant polynomials must be handled separately when applying the lemma to ramification ideals.
\end{Obs}

For a nontrivial subgroup $H$ of $\gal(L/K,v)$ and a generator $\theta$ of $L/K$ we say that $(H,\theta)$ satisfies  \textbf{(HC)} if for every $a\in K_H$ and every $\sigma\in H$ we have
\[
v\left(\frac{\sigma\theta-\theta}{\theta-a}\right)\geq 0.
\]
\begin{Obs}
If $(K(\theta)/K,v)$ is a unibranched immediate Galois extension, then for each nontrivial subgroup $H$ of $\gal(L/K)$ it follows from \cite[Lemma 3.2 (1)]{KR} that
\[
v\left(\frac{\sigma \theta-\theta}{\theta-a}\right)>0\mbox{ for every }a\in K_H\mbox{ and every }\sigma\in H.
\]
In particular, $(H,\theta)$ satisfies \textbf{(HC)}.
\end{Obs}
\begin{Teo}\label{generalizakuhlm}
Assume that $(L/K,v)$, with $L=K(\theta)$, is a finite immediate Galois extension and that $v$ is invariant under $\gal(L/K)$. For each nontrivial subgroup $H$ of $\gal(L/K)$ such that $D_1(\theta,K_H)$ is bounded above, we have
\begin{equation}\label{comparisset}
I_{S_H}\subseteq I_H. 
\end{equation}
Moreover, if $(K,v)$ has rank one and characteristic $p>0$, $1<|H|\leq p$, and $(H,\theta)$ satisfies \textbf{(HC)}, then $D_1(\theta,K_H)$ is bounded above and the above inclusion is an equality of sets; thus boundedness need not be assumed separately in this case.
\end{Teo}
\begin{proof}
Choose $\sigma_0\in H\setminus\{id\}$ such that
\[
v(\sigma_0\theta-\theta)=\min S(\theta,H).
\]
For every $a\in K_H$, the element $\theta-a$ is nonzero and
\[
v\left(\frac{\sigma_0(\theta-a)-(\theta-a)}{\theta-a}\right)
=\min S(\theta,H)-v(\theta-a).
\]
Thus every element of $S_H$ is a logarithmic ramification value for $\sigma_0$, proving $I_{S_H}\subseteq I_{\sigma_0}\subseteq I_H$. This argument does not require rank one.

For the reverse inclusion, assume the additional hypotheses. Since $L/K$ is immediate, so is $L/K_H$; in particular, $D_1(\theta,K_H)$ has no maximum. The field $K_H$ has rank one and the same residue characteristic exponent $p$ as $K$. Fix any $\tau\in H\setminus\{id\}$. Since $\theta$ generates $L/K$, $\tau\theta\ne\theta$, and \textbf{(HC)} gives
\[
v(\theta-a)\leq v(\tau\theta-\theta)<\infty
\qquad\text{for every }a\in K_H.
\]
Hence $D_1(\theta,K_H)$ is bounded above. All the field and approximation hypotheses of Lemma \ref{lemakapland} are therefore satisfied over $K_H$.

Take $b\in L^*$. Write $b=f(\theta)$ with $f\in K_H[x]$ and $\deg f<[L:K_H]=|H|\leq p$. If $f$ is constant, then $\sigma b-b=0$ for every $\sigma\in H$, so $b$ contributes no nonzero generator to $I_\sigma$. Likewise, $I_{id}=\{0\}$. We may therefore assume that $f$ is nonconstant and $\sigma\in H\setminus\{id\}$.

For $1\leq\deg f<p$, the derivative-stability hypothesis is automatic by Lemma \ref{lemakapland}. Choose $a\in K_H$ sufficiently close to $\theta$ so that both conclusions of that lemma hold for $f$. In particular,
\begin{equation}\label{equation1parafoas}
v\left(\partial_1f(\theta)(\theta-a)\right)
<v\left(\partial_i f(\theta)(\theta-a)^i\right)
\quad(2\leq i\leq\deg f).
\end{equation}
Zero derivatives on the right have value $\infty$. Put $\eta=v(\theta-a)$ and $s=v(\sigma\theta-\theta)$. Condition \textbf{(HC)} gives $s-\eta\geq0$, and hence
\begin{equation}\label{equation2parafoas}
(v(\partial_i f(\theta))+is)
-(v(\partial_1 f(\theta))+s)
> (i-1)(s-\eta)\geq0
\end{equation}
for each nonzero derivative of order $i>1$; the comparison is immediate for zero derivatives. Thus the first-order term also uniquely dominates in
\[
f(\sigma\theta)-f(\theta)
=\sum_{i=1}^{\deg f}\partial_i f(\theta)(\sigma\theta-\theta)^i.
\]
It follows that
\begin{equation}\label{anothequautl}
v(\sigma b-b)=v(\partial_1 f(\theta))+s.
\end{equation}
The final assertion of Lemma \ref{lemakapland} gives
\begin{equation}\label{anothequaut2}
v(b)=v(f(a))<v(\partial_1 f(\theta))+\eta.
\end{equation}
Subtracting, we obtain
\[
v\left(\frac{\sigma b-b}{b}\right)
>s-\eta\geq\min S(\theta,H)-v(\theta-a)\in S_H.
\]
Therefore every logarithmic ramification value for $\sigma$ lies above an element of $S_H$, so $I_\sigma\subseteq I_{S_H}$. Taking the union over $\sigma\in H$ proves $I_H\subseteq I_{S_H}$ and hence equality.
\end{proof}

\subsection{Compositum of defect AS extensions}\label{sectioncomposiu}
In this section we present some results that follow from \cite{NNASD} and the discussions above. 

\begin{Def}
We say that the valued field $(K,v)$, of positive characteristic $p$, satisfies the condition \textbf{(GE)} if for every $n\in\N$ there exist  $c_1,\ldots,c_n\in K$ such that
\begin{equation}\label{equationinte}
v(a_1c_1+\ldots+a_nc_n)=0\mbox{ for every }(a_1,\ldots,a_n)\in \F_p^n\setminus\{(0,\ldots,0)\}.
\end{equation}
\end{Def}

\begin{Obs}
If $K$ admits an infinite subfield $K_0$ which is algebraic over $\F_p$, then $(K,v)$ satisfies the condition \textbf{(GE)}. 
\end{Obs}

\begin{Prop}\label{ramificiudle}
Assume that $(K,v)$ is Henselian, has rank one, and satisfies \textbf{(GE)}. Take defect AS elements $\alpha_1,\ldots,\alpha_n\in\overline K$, linearly disjoint over $K$, and assume that their compositum $L=K(\alpha_1,\ldots,\alpha_n)$ is immediate over $K$. For each $i$, $1\leq i\leq n$, set
\[
K_i=K(\alpha_1,\ldots,\alpha_{i-1},\alpha_{i+1},\ldots,\alpha_n)\mbox{ and }\textbf D_i=D_1(\alpha_i,K_i).
\]
For every $i$, $1\leq i\leq n$, there exists a subgroup $H$ of $\gal(L/K)$ such that $I_H=I_{-\textbf D_i}$.
\end{Prop}
\begin{proof}
It follows from \cite{NNASD} that $\gal(L/K)\simeq C_p\times\ldots\times C_p$ is generated by $\sigma_i$, $1\leq i\leq n$, where
\[
\sigma_i(\alpha_i)=\alpha_i+1\mbox{ and }\sigma_i(\alpha_j)=\alpha_j\mbox{ if }i\neq j.
\]

Take $c_1,\ldots,c_n\in K$ such that \eqref{equationinte} is satisfied and set $\theta=c_1\alpha_1+\ldots+c_n\alpha_n$. For each $i$, $1\leq i\leq n$, consider the subgroup $H=\langle \sigma_i\rangle$ of $\gal(L/K)$. Then $K_i=K_H$ and $|H|=p$. On the other hand, since
\[
\theta-c_i\alpha_i\in K_i=K_H
\]
and $v(c_i)=0$, we obtain that
\[
D_1(\theta,K_H)=D_1(c_i\alpha_i,K_H)=D_1(\alpha_i,K_H).
\]
Moreover, $\min S(\theta,H)=v(c_i)=0$. Since $L/K$ is immediate, Theorem \ref{generalizakuhlm} applies and implies that
\[
I_H=I_{-\textbf D_i}.
\]
\end{proof}

\begin{Cor}\label{linealrydisoj}
Suppose that $(K,v)$ is Henselian, of characteristic $p>0$, and satisfies \textbf{(GE)}. Let $n\geq1$, and let $K(\alpha_1),\ldots,K(\alpha_n)$ be a linearly disjoint family of defect AS extensions such that $d_1(\alpha_i)=0^-$ for every $i$. Set $L=K(\alpha_1,\ldots,\alpha_n)$ and $\mathcal E=(L/K,v)$. Assume, in addition, that the compositum $(L/K,v)$ is immediate. Then
\[
{\rm Ram}(\mathcal E)=\{\mathcal M_L\}.
\]
\end{Cor}
\begin{proof}
Since $(K,v)$ is Henselian and $(L/K,v)$ is immediate, by \cite[Lemma 3.5]{Topics} we have, for every $\sigma\in\gal(L/K)$,
\[
v\left(\frac{\sigma b-b}{b}\right)>0\mbox{ for every }b\in L^*. 
\]
In particular, $I_H\subseteq\mathcal M_L$ for every subgroup $H$ of $\gal(L/K)$. Here $\sigma b-b=0$ is interpreted as having value $\infty$; in particular, $I_{id}=\{0\}$.

Take $c_1,\ldots,c_n\in K$ satisfying \eqref{equationinte} and set $\theta=c_1\alpha_1+\ldots+c_n\alpha_n$. In particular, $v(c_i)=0$ for each $i$. Every nonidentity $\sigma\in\gal(L/K)$ acts by $\sigma\alpha_i=\alpha_i+j_i$, for a nonzero vector $(j_1,\ldots,j_n)\in\F_p^n$. Thus
\[
v(\sigma\theta-\theta)=v\left(\sum_{i=1}^n j_i c_i\right)=0.
\]
In particular, $\theta$ generates $L/K$. For every nonidentity $\sigma\in\gal(L/K)$ and every $a\in K$, we therefore have
\[
v\left(\frac{\sigma(\theta-a)-(\theta-a)}{\theta-a}\right)=-v(\theta-a).
\]
For every $\delta\in\Gamma$ with $\delta<0$, choose $a_i\in K$ such that $v(\alpha_i-a_i)>\delta$, using $d_1(\alpha_i)=0^-$. Then $a=\sum_i c_i a_i\in K$ satisfies
\[
v(\theta-a)\geq\min_i v(c_i(\alpha_i-a_i))>\delta.
\]
Now fix a nontrivial subgroup $H$ and $\sigma\in H\setminus\{id\}$. For any $x\in\mathcal M_L\setminus\{0\}$, take $\delta=-v(x)<0$ and choose $a$ as above. The displayed logarithmic valuation is then
\[
v\left(\frac{\sigma(\theta-a)-(\theta-a)}{\theta-a}\right)
=-v(\theta-a)<v(x),
\]
so $x\in I_\sigma\subseteq I_H$. Hence $\mathcal M_L\subseteq I_H$, and equality follows for every nontrivial $H$.
\end{proof}

\subsection{Examples from \cite{NNASD}} \label{computationcassesAsd} For this section we will denote the Hahn field $\overline{\F_p}((t^\Q))$ by $\mathbb H$. Let $K$ be a subfield of $\mathbb H$, containing $\overline{\F_p}(t)$, which is perfect, Henselian and has no finite limits. Take elements $\alpha,\theta,\eta\in \mathbb H$ such that
\[
AS(\alpha)=t^{-1}, AS(\theta)=\alpha\mbox{ and } AS(\eta)=\alpha^2.
\]
We choose their negative-support representatives, namely
\[
\alpha=\sum_{j\geq1}t^{-1/p^j},\qquad
\theta=\sum_{j\geq1}\alpha^{1/p^j}
=\sum_{j\geq1}j\,t^{-1/p^{j+1}},\qquad
\eta=\sum_{j\geq1}\alpha^{2/p^j},
\]
where integer coefficients are read in $\F_p$ and all sums are Hahn sums. These choices do not change the generated fields, since any two roots of the same AS equation differ by an element of $\F_p$. Except for the separate characteristic-two discussion below, we assume $p>2$ throughout this subsection.
Take $c\in \overline{\F_p}\subseteq K$ such that $c^p-c=1$. Take $\gamma,\omega\in \mathbb H$ defined by
\[
\gamma=\theta-c\alpha\mbox{ and }\omega=\eta-\alpha(\theta+\gamma).
\]
Set
\[
L'=K(\gamma)\mbox{ and }L=K(\alpha)
\]
and
\[
M'=K(\omega), M_0=K(\theta)\mbox{ and }M=K(\eta).
\]
Finally, set $N=K(\theta,\eta)$.
For $p>2$, these fields fit in the following diagram; its degree-$p$ tower structure is not asserted in characteristic two.
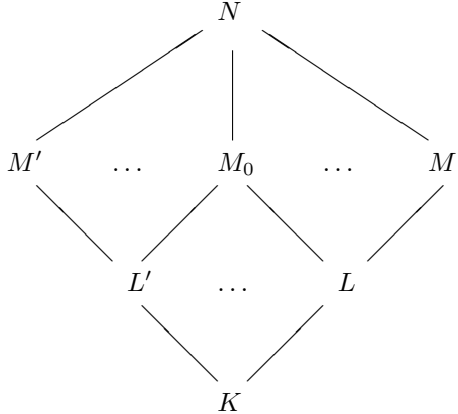
\begin{center}
	\setlength{\unitlength}{4mm}
	\begin{picture}(12.4,14)

\put(7,13){\begin{footnotesize}$N$\end{footnotesize}}

\put(0,8){\begin{footnotesize}$M'$\end{footnotesize}}
\put(3.5,8){\begin{footnotesize}$\ldots$\end{footnotesize}}
\put(7,8){\begin{footnotesize}$M_0$\end{footnotesize}}
\put(10.5,8){\begin{footnotesize}$\ldots$\end{footnotesize}}
\put(14,8){\begin{footnotesize}$M$\end{footnotesize}}

\put(4,4){\begin{footnotesize}$L'$\end{footnotesize}}
\put(7,4){\begin{footnotesize}$\ldots$\end{footnotesize}}
\put(11,4){\begin{footnotesize}$L$\end{footnotesize}}

\put(7,0){\begin{footnotesize}$K$\end{footnotesize}}

		\put(4.5,5){\line(1,1){2.5}}
		\put(10.5,5){\line(-1,1){2.5}}
		\put(7,1){\line(-1,1){2.5}}
		\put(8,1){\line(1,1){2.5}}
		\put(3.5,5){\line(-1,1){2.5}}
		\put(12,5){\line(1,1){2.5}}
		
		\put(1,9){\line(3,2){5.5}}
		\put(14,9){\line(-3,2){5.5}}
		\put(7.5,9){\line(0,1){3}}
	\end{picture}
\end{center}\bs
The next result summarizes some results from \cite{NNASD}.
\begin{Prop}
Assume $p>2$. With the field extensions described above, we have:
\begin{description}
\item[(i)] All the degree $p$ extensions above are defect AS extensions (hence are Galois).
\item[(ii)] The extensions $N/K$ and $M_0/K$ are Galois, but $M'/K$ and $M/K$ are not.
\item[(iii)] ${\rm Gal}(M_0/K)\simeq C_p\times C_p$.
\item[(iv)] ${\rm Gal}(N/K)\simeq\left(C_p\times C_p\right)\rtimes C_p$, with nontrivial action.
\end{description}
\end{Prop}

We suppose now that $p>2$. We will consider the elements $\sigma,\iota,\tau\in \gal(N/K)$ defined by
\begin{displaymath}
\begin{array}{cccc}
\sigma(\theta)=\theta+c &\sigma(\eta)=\eta+2\theta+c &\sigma(\omega)=\omega&\sigma(\alpha)=\alpha+1\\[10pt]
\iota(\theta)=\theta &\iota(\eta)=\eta+1 &\iota(\omega)=\omega+1&\iota(\alpha)=\alpha\\[10pt]
\tau(\theta)=\theta+1 &\tau(\eta)=\eta &\tau(\omega)=\omega-2\alpha&\tau(\alpha)=\alpha
\end{array}.
\end{displaymath}
\begin{Lema}
We have the following.
\begin{description}
\item[(i)] $\iota$ commutes with $\sigma$ and $\tau$.
\item[(ii)] $\sigma\tau\sigma^{-1}=\iota^{-2}\tau$.
\item[(iii)] $\sigma\tau\sigma^{-1}\tau^{-1}=\iota^{-2}\Lr \tau\sigma^{-1}\tau^{-1}=\iota^{-2}\sigma^{-1}$.
\item[(iv)] $\tau^{-1}\sigma\tau\sigma^{-1}=\iota^{-2}\Lr \tau^{-1}\sigma\tau=\iota^{-2}\sigma$.
\item[(v)] $\sigma\tau=\iota^{-2}\tau\sigma$.
\end{description}
\end{Lema}
For simplicity, for any of the fields of the diagram we will denote only by $\mathcal M$ the corresponding maximal ideal.
\begin{Obs}
All the extensions that appear in the diagram are immediate and unibranched. Indeed, $K$ is Henselian, its residue field is $\overline{\F_p}$, and its value group is $p$-divisible because $K$ is perfect. For these extensions of $p$-power degree, the ramification index and residue degree are therefore both one. Consequently, all logarithmic ramification values for nonidentity automorphisms are positive (or $\infty$). For an extension $E/F$ in the diagram, it suffices to exhibit, for each nonidentity $\rho$ under consideration, elements $b_n\in E^*$ such that
\[
0<v\left(\frac{\rho b_n-b_n}{b_n}\right)\longrightarrow0.
\]
Then $I_\rho=\mathcal M_E$, and any subgroup containing $\rho$ has the same ramification ideal.
\end{Obs}

We discuss now what happens with the Galois extensions of degree $p$ in the diagram. The following result follows from \cite{KR}. We present a short proof here for sake of completeness.
\begin{Prop}
The extensions $(L/K,v)$ and $(L'/K,v)$ have only one ramification ideal and it is the maximal ideal.
\end{Prop}
\begin{proof}
For an element in $\gal(N/K)$ we will denote its restriction to $L'$ or $L$ by the same symbol. First take $G=\gal(L/K)=\langle\sigma\rangle$. For any $a\in K$ we have
\[
v\left(\frac{\sigma (\alpha-a)-(\alpha-a)}{\alpha-a}\right)=v\left(\frac{1}{\alpha-a}\right)=-v(\alpha-a).
\]
Since $d_1(\alpha)=0^-$ we deduce that $I_{\langle\sigma\rangle}=\mathcal M$.

For $L'=K(\gamma)$, write the Hahn expansion
\[
\gamma=\theta-c\alpha
=\sum_{m\geq1}(m-1-c)t^{-1/p^m},
\]
where the integer $m-1$ is read in $\F_p$. Since $c^p-c=1$, we have
$c\notin\F_p$, and hence every coefficient $m-1-c$ is nonzero. The finite
truncations belong to $K$ and approximate $\gamma$ cofinally below zero,
whereas the no-finite-limits hypothesis prevents an element of $K$ from
having the whole negative support. Thus $d_1(\gamma)=0^-$. The restriction
of $\tau$ to $L'$ satisfies $\tau(\gamma)=\gamma+1$, so the same calculation
gives $I_{\langle\tau\rangle}=\mathcal M$.
\end{proof} 

We proceed now with the Galois extension of degree $p^2$ in the diagram.
\begin{Prop}
The extension $M_0/K$ admits only one ramification ideal and it is the maximal ideal.
\end{Prop}
\begin{proof}
For an element in $\gal(N/K)$ we will denote its restriction to $M_0$ by the same symbol. We have $M_0=K(\theta)$ and
\[
\gal(M_0/K)=\langle\sigma,\tau\rangle\simeq C_p\times C_p.
\]
Since $d_1(\theta)=0^{-}$, we have $I_G=\mathcal M$ is the largest ramification ideal. For $a\in K$ we have
\begin{equation}\label{equasigma}
v\left(\frac{\sigma(\theta-a)-(\theta-a)}{\theta-a}\right)=v\left(\frac{\sigma\theta-\theta}{\theta-a}\right)=v\left(\frac{c}{\theta-a}\right)=-v(\theta-a).
\end{equation}
We also have
\begin{equation}\label{equatau}
v\left(\frac{\tau(\theta-a)-(\theta-a)}{\theta-a}\right)=v\left(\frac{\tau\theta-\theta}{\theta-a}\right)=v\left(\frac{1}{\theta-a}\right)=-v(\theta-a).
\end{equation}
By what we said before, we deduce from \eqref{equasigma} and \eqref{equatau} that
\[
I_{\langle\sigma\rangle}=\mathcal M=I_{\langle\tau\rangle}.
\]

Finally, all the other subgroups of degree $p$ of $G$ are of the form $\langle \sigma^i\tau\rangle$ or $\langle\sigma\tau^j\rangle$. In all those cases, we have $I_H=\mathcal M$ because
\[
(\sigma^i\tau)(\theta)=\theta + 1+ ic\mbox{ and }\sigma\tau^j(\theta)=\theta+c+j.
\]
\end{proof}

\begin{Obs}\label{char-two-tower}
The equations defining $M_0$ have a different consequence in characteristic two. With the same hypotheses on $K$ and the negative-support representatives chosen above, they give
\[
\theta^4+\theta=t^{-1},\qquad
\theta=\sum_{j\geq1}t^{-1/4^j}.
\]
Since $\F_4\subseteq K$, all roots are $\theta+a$ with $a\in\F_4$. Thus $K(\theta)/K$ is Galois and its degree divides four. The element $\theta$ does not belong to $K$, because it has infinite negative support. If its degree were two, the nonidentity automorphism would send $\theta$ to $\theta+a$ for some $a\in\F_4^*$, and its norm would give $\theta^2+a\theta\in K$. However, the supports of $\theta^2$ and $\theta$ are respectively
\[
\{-1/2^{2j-1}:j\geq1\},\qquad
\{-1/2^{2j}:j\geq1\}.
\]
They are disjoint and infinite, so $\theta^2+a\theta$ cannot belong to $K$. Therefore
\[
[M_0:K]=4,\qquad\gal(M_0/K)\simeq(\F_4,+)\simeq C_2\times C_2,
\]
not $C_4$.

Since $K$ is perfect, the finite truncations
\[
\theta_n=\sum_{j=1}^n t^{-1/4^j}\in K,
\qquad v(\theta-\theta_n)=-1/4^{n+1},
\]
show that $d_1(\theta)=0^-$; the absence of a $K$-approximation of value at least zero follows again from infinite negative support. If $\epsilon\in\overline K$ satisfies $v(\theta-\epsilon)\geq0$, the residue of $\epsilon-\theta$ belongs to $k=\overline{\F_2}\subseteq K$. Choose $r\in k$ representing it. Then
\[
v\bigl(\theta-(\epsilon-r)\bigr)>0=\omega(\theta).
\]
Krasner's lemma, with this strict inequality and Henselian $K$, gives $K(\theta)\subseteq K(\epsilon-r)=K(\epsilon)$. Hence $[K(\epsilon):K]\geq4$, and
\[
\dep(\theta)=\dep(M_0/K,v)=1.
\]
The extension $M_0/K$ is immediate by the same value-group and residue-field argument used above. For every nonidentity translation $\rho_a(\theta)=\theta+a$, $a\in\F_4^*$, we have
\[
v\left(\frac{\rho_a(\theta-\theta_n)-(\theta-\theta_n)}{\theta-\theta_n}\right)=1/4^{n+1}\longrightarrow0.
\]
Thus ${\rm Ram}(M_0/K,v)=\{\mathcal M_{M_0}\}$, but its cardinality equals, rather than is smaller than, the depth.

Moreover, $AS(\theta+\alpha)=\alpha^2$, and for the chosen representatives $\eta=\theta+\alpha=\theta^2$. Since $\theta=\eta^2+t^{-1}$, this gives $M=M_0=N$. In particular, the degree-$p^3$ diagram does not extend to $p=2$. The cyclic degree-four example of \cite[Lemma 4.4]{NNASD} uses the different equation $AS(\eta)=t^{-1}\alpha$ and is not the present $M_0$.
\end{Obs}
We return to $p>2$ and the degree-$p^3$ extension $N/K$ in the diagram.

\begin{Prop}\label{ramification-N}
Assume $p>2$. For every nonidentity $\rho\in\gal(N/K)$, we have $I_\rho=\mathcal M_N$. Consequently,
\[
{\rm Ram}(N/K,v)=\{\mathcal M_N\}.
\]
\end{Prop}
\begin{proof}
The extension is immediate and unibranched, so $I_\rho\subseteq\mathcal M_N$ for every $\rho$. We construct positive logarithmic ramification values tending to zero for each nonidentity automorphism.

First, set
\[
\beta=t^{-2/p}+2t^{-1/p}\alpha^{1/p}\in L.
\]
Since $\alpha=t^{-1/p}+\alpha^{1/p}$, we have
\[
\beta=\alpha^2-\alpha^{2/p},\qquad
v(\beta)=\min\left\{-\frac2p,-\frac1p-\frac1{p^2}\right\}=-\frac2p.
\]
The field $L$ is perfect, being algebraic over the perfect field $K$, so every $\beta^{1/p^j}$ belongs to $L$. Telescoping the Hahn sums gives
\[
\eta=\sum_{j\geq1}\alpha^{2/p^j}
=\sum_{j\geq1}j\,\beta^{1/p^j},
\]
where $j$ is interpreted in $\F_p$. In detail,
\[
\sum_{j=1}^{n-1}j\,\beta^{1/p^j}
=\sum_{j=1}^{n-1}\alpha^{2/p^j}-(n-1)\alpha^{2/p^n},
\]
and the remainder term on the right tends coefficientwise to zero as a Hahn series.
For $n\geq1$, put
\[
\beta_n=\sum_{j=1}^{n-1}j\,\beta^{1/p^j}\in L,
\qquad m(n)=\min\{j\geq n:p\nmid j\}.
\]
The first nonzero term of the tail has strictly smaller valuation than all subsequent terms. Therefore
\begin{equation}\label{eta-tail-value}
v(\eta-\beta_n)=-\frac{2}{p^{m(n)+1}}.
\end{equation}
For $d\in\F_p^*$, $\iota^d$ fixes $L$ and sends $\eta$ to $\eta+d$. Hence
\[
v\left(\frac{\iota^d(\eta-\beta_n)-(\eta-\beta_n)}{\eta-\beta_n}\right)
=\frac{2}{p^{m(n)+1}}\longrightarrow0.
\]
This proves $I_{\iota^d}=\mathcal M_N$ for every nonzero $d$.

For the remaining automorphisms, use the finite approximants
\[
\theta_n=\sum_{j=1}^{n-1}j\,t^{-1/p^{j+1}}\in K.
\]
The explicit expansion of $\theta$ gives
\[
v(\theta-\theta_n)=-\frac1{p^{m(n)+1}}\longrightarrow0.
\]
The commutation relations and the order $p^3$ of the group give a unique expression
\[
\rho=\sigma^a\tau^b\iota^d,\qquad a,b,d\in\{0,\ldots,p-1\}.
\]
If $(a,b)\ne(0,0)$, then
\[
\rho(\theta)-\theta=ac+b\ne0,
\]
since $c\notin\F_p$. This difference is a nonzero constant and has value zero. As $\rho$ fixes $\theta_n$, we obtain
\[
v\left(\frac{\rho(\theta-\theta_n)-(\theta-\theta_n)}{\theta-\theta_n}\right)
=\frac1{p^{m(n)+1}}\longrightarrow0.
\]
Thus $I_\rho=\mathcal M_N$ in this case too. If $a=b=0$ and $\rho\ne id$, then $d\ne0$, which is the central case already treated. Every nontrivial subgroup contains a nonidentity element, so its ramification ideal is also $\mathcal M_N$.
\end{proof}

\begin{Obs}\label{depn3famif1}
For the present tower in characteristic $p>2$, \cite[Section 4.4]{NNASD} proves $\dep(N/K,v)=2$. Combining that depth result with Proposition \ref{ramification-N}, we obtain the example
\[
\#{\rm Ram}(N/K,v)=1<2=\dep(N/K,v).
\]
Thus this inequality is supplied by the nonabelian degree-$p^3$ extension, not by the characteristic-two field $M_0$ discussed in Remark \ref{char-two-tower}.
\end{Obs}

\end{document}